\documentclass[12pt,leqno,a4paper] {amsart}
\usepackage[utf8]{inputenc}
\usepackage{amssymb,enumerate,blkarray,rotating}
\usepackage{multirow}

\usepackage{tikz}

\newcommand{\QQ}{{\mathbb{Q}}}
\newcommand{\FF}{{\mathbb{F}}}
\newcommand{\ZZ}{{\mathbb{Z}}}

\newcommand{\bB} {\mathbf B}
\newcommand{\bC} {\mathbf C}
\newcommand{\bG} {\mathbf G}
\newcommand{\bL} {\mathbf L}
\newcommand{\bM} {\mathbf M}
\newcommand{\bP} {\mathbf P}
\newcommand{\bT} {\mathbf T}
\newcommand{\bU} {\mathbf U}
\newcommand{\bZ} {\mathbf Z}

\newcommand{\cL} {\mathcal L}

\newcommand{\Ind}{{\operatorname{Ind}}}
\newcommand{\Irr}{{\operatorname{Irr}}}
\newcommand{\uni}{{\operatorname{uni}}}
\newcommand{\Trace}{{\operatorname{Trace}}}
\newcommand{\GL}{\operatorname{GL}}
\newcommand{\SL}{\operatorname{SL}}
\newcommand{\SU}{\operatorname{SU}}
\newcommand{\SO}{\operatorname{SO}}
\newcommand{\Spin}{{\operatorname{Spin}}}
\newcommand{\QLG}{{Q_\bL^\bG}}
\newcommand{\RLG}{{R_\bL^\bG}}
\newcommand{\RTL}{{R_\bT^\bL}}
\newcommand{\RTG}{{R_\bT^\bG}}

\newcommand{\tbG}{{\tilde\bG}}
\newcommand{\tbL}{{\tilde\bL}}

\newcommand{\tQ}{{\tilde Q}}
\newcommand{\tQLG}{{\tilde Q_\bL^\bG}}
\newcommand{\tRTG}{{\tilde R_\bT^\bG}}
\newcommand{\tRTL}{{\tilde R_\bT^\bL}}
\newcommand{\hlf}{\frac{1}{2}}
\newcommand{\pl}{{\!+\!}}
\newcommand{\mn}{{\!-\!}}
\newcommand{\tw}[1]{{}^{#1}\!}
\newcommand{\Ph}[1]{\Phi_{#1}}

\let\eps=\epsilon

\let\la=\lambda

\def\pmod#1{~({\rm mod}~#1)}

\newtheorem{thm}{Theorem}[section]
\newtheorem{lem}[thm]{Lemma}
\newtheorem{cor}[thm]{Corollary}
\newtheorem{prop}[thm]{Proposition}
\newtheorem{conj}[thm]{Conjecture}

\theoremstyle{definition}
\newtheorem{rem}[thm]{Remark}

\newtheorem{exmp}[thm]{Example}

\begin{document}

\title[Green functions for spin groups]{The 2-parameter Green functions\\ for 8-dimensional spin groups}
\date{\today}
\author{Gunter Malle}
\address{FB Mathematik, TU Kaiserslautern, Postfach 3049,
  67653 Kaisers\-lautern, Germany.}
\email{malle@mathematik.uni-kl.de}
\author{Emil Rotilio}
\address{FB Mathematik, TU Kaiserslautern, Postfach 3049,
  67653 Kaisers\-lautern, Germany.}
\email{rotilio@mathematik.uni-kl.de}

\thanks{The authors gratefully acknowledge financial support by SFB TRR 195.}

\keywords{2-parameter Green function, character formula, disconnected centre, 
  $\Spin_8^+(q)$}

\subjclass[2010]{20C15, 20C33, 20G40}

\begin{abstract}
The 2-parameter Green functions occur as a crucial ingredient in the character
formula for Lusztig induction in finite reductive groups. Still, very little is
known about these functions, in particular in the case of groups arising
from algebraic groups with disconnected centre. We collect some basic
properties and then apply these, together with some explicit computations, to
determine all 2-parameter Green functions for 8-dimensional spin groups in odd
characteristic, whose centre is disconnected of order~4.
\end{abstract}

\maketitle


\section{Introduction}
The 2-parameter Green functions are a central ingredient in the character
formula for Lusztig induction and restriction in finite groups of Lie type.
Despite their importance, little seems to be known about their values.
We collect some basic properties and compute these functions for several
families of groups of small rank, with connected and non-connected centre.

Our main result is the complete determination of the 2-parameter Green functions
for the family of spin groups $\Spin_8^+(q)$ for odd $q$, see
Theorem~\ref{thm:Spin8}, which we give in the form of tables at the end of
this paper.

Our results also lead us to formulate two conjectures on the values of
2-parameter Green functions, the first (Conjecture~\ref{conj:integral}) for
arbitrary finite reductive groups, the second (Conjecture~\ref{conj:type A})
particular for groups of type~$A$.

\section{2-parameter Green functions}

\subsection{Definition and first properties}
Let $\bG$ be connected reductive with a Steinberg map $F:\bG\to\bG$. For an
$F$-stable Levi subgroup $\bL$ of a parabolic subgroup $\bP$ of $\bG$ with Levi
decomposition $\bP=\bU\bL$,
$$\QLG\!:\!\bG_\uni^F\times\bL_\uni^F\rightarrow\QQ,\ (u,v)\mapsto
  \frac{1}{|\bL^F|}\sum_{i\ge0}(-1)^i\Trace\big((u,v)\!\mid
  \text{H}_c^i(\cL^{-1}(\bU))\big),\!$$
is called the associated \emph{2-parameter Green function}.
Here $\cL:\bG\to\bG$, $g \mapsto g^{-1}F(g)$, denotes the Lang map,
$\bG_\uni$ is the set of unipotent elements of $\bG$, and $\text{H}_c^i$ is
$\ell$-adic cohomology with compact support, for a prime $\ell$ different
from the characteristic of~$\bG$.

By an abuse of notation we omit the parabolic subgroup $\bP$ from our
notation. This is justified as by the Mackey formula for Lusztig induction,
the 2-parameter Green function is independent of $\bP$ in all situations
that we consider here. It follows in particular that $\QLG(u,v_1)=\QLG(u,v_2)$
whenever $v_1,v_2\in\bL^F$ are conjugate in $N_\bG(\bL)^F$. 
Not much else seems to be known about the $\QLG$.

The relevance of 2-parameter Green functions comes from the fact that they occur
in the character formula for Lusztig induction (see \cite[Thm.~2.2]{DiMi0} and
\cite[Prop.~3.2]{DiMi1}):

\begin{thm}   \label{thm:charform}
 Let $\bL$ be an $F$-stable Levi subgroup of $\bG$. Then
 $$\RLG(\psi)(g)=
   \frac{1}{|\bL^F|\,|C_\bG^\circ(s)^F|}\sum_{\{h\in\bG^F\mid s^h\in\bL\}}
    \!\!\!|C_{\tw{h}\bL}^\circ(s)^F|\!\!\sum_{v\in C_{\tw{h}\bL}^\circ(s)_\uni^F}
    \!\!\!Q_{C_{^h\!\bL}^\circ(s)}^{C_\bG^\circ(s)}(u,v^{-1})\,{}^h\!\psi(sv),
 $$
 for $\psi\in\Irr(\bL^F)$ and $g\in\bG^F$ with Jordan decomposition $g=su$.
\end{thm}

If $\bL=\bT$ is a maximal torus, so $\bP$ is a Borel subgroup of $\bG$, then
$\bL_\uni^F=\{1\}$ and the defining formula shows that
$Q_\bT^\bG(u,1)=R_{\bT}^{\bG}(1)(u)$ $(u\in\bG_\uni^F)$, which is
the usual (1-parameter) Green function.
The values of $\QLG(u,v)$ at $u=1$ are known for any $\bL$, see
\cite[p.~157]{DiMi2}:
\begin{equation}   \label{eq1}
  \QLG(1,v)=\begin{cases}
  \eps_\bG\eps_\bL|\bG^F:\bL^F|_{p'}& \text{if $v=1$,}\\ 0& \text{otherwise.}
  \end{cases}
\end{equation}

If $\bL$ is a split Levi subgroup of $\bG$, that is, if we can choose $\bP$ and
hence its unipotent radical $\bU$ to be $F$-stable, then
\begin{equation}   \label{eq2}
  \QLG(u,v)=\frac{1}{|\bL^F|}|\{g\bU^F\mid g\in\bG^F,\,u^g\in v\bU^F\}|
\end{equation}
can be computed in an elementary way. This shows:

\begin{prop}   \label{prop:split case}
 Assume that $\bP$ is an $F$-stable parabolic subgroup of $\bG$ with Levi
 decomposition $\bP=\bU.\bL$. Then for $v\in\bL_\uni^F$, $u\in\bG_\uni^F$ we
 have:
 \begin{enumerate}
  \item[\rm(a)] $|v^{\bL^F}|\,\QLG(u,v)\in\ZZ_{\ge0}$.
  \item[\rm(b)] If $\QLG(u,v^{-1})\ne0$ then
   $v^\bG\subseteq\overline{u^\bG}\subseteq\overline{\Ind_\bL^\bG(v^\bL)}$.
  \item[\rm(c)] If $u$ is regular unipotent, then there is a unique
   $\bL^F$-class
   $C$ of regular unipotent elements of $\bL^F$ such that
   $$\QLG(u,v)=\begin{cases}
         |v^{\bL^F}|^{-1}& \text{if $v\in C$,}\\ 0& \text{otherwise}.
    \end{cases}$$
 \end{enumerate}
\end{prop}

Here, $\overline{\bC}$ denotes the closure of a class $\bC$ and
$\Ind_\bL^\bG\bC$ is the induced class in the sense of Lusztig--Spaltenstein
\cite{LS79}.

\begin{proof}
Observe that if $g\in\bG^F$ is such that $u^g\in v\bU^F$ then for any
$c\in C_\bL(v)^F$, the element $gc$ has the same property, so
$|\{g\bU^F\mid u^g\in v\bU^F\}|$ is divisible by $|C_\bL(v)^F|$, whence
(\ref{eq2}) shows that $|v^{\bL^F}|\QLG(u,v)$ is an integer. 
\par
If $\QLG(u,v^{-1})\ne0$ there is $g\in\bG^F$ with $u^g\in v\bU^F$, so up
to replacing $u$ by a conjugate, we have $u\in v\bU^F$. Now by definition the
induced class $\bC:=\Ind_\bL^\bG(v^\bL)$ has the property that
$\bC\cap v^\bL\bU$ is dense in $v^\bL\bU$. Hence, $u\in\overline{\bC}$.
Moreover, we have $u=v x$ for some $x\in\bU^F$. Then
$X:=\{vx^c\mid c\in Z^\circ(\bL)\}\subseteq u^\bG$.
Now $Z^\circ(\bL)$ acts non-trivially on all root subgroups of $\bU$ as
$\bL=C_\bG(Z^\circ(\bL))$. Thus the closure of $X$ contains $v$ and so
$v\in\overline{u^\bG}$.
\par
For~(c) note that the centraliser dimension of $v$ in $\bL$ and of
$v'\in\Ind_\bL^\bG(v^\bL)$ in $\bG$ agree (see \cite[Thm.~1.3(a)]{LS79}).
It follows that only the regular unipotent class of $\bL$ induces to the regular
unipotent class of $\bG$, and thus $\QLG(u,v)=0$ unless $v$ is regular. Now
assume that $u\in v\bU^F$, and so $\QLG(u,v)\ne0$. Since $u$ is regular, it lies
in a unique Borel subgroup $\bB\le\bP$ of $\bG$. Thus, if $g\in\bG^F$ with
$g^{-1}ug\in\bP$ then $g\in\bP^F$. In particular, $g^{-1}ug\in v'\bU^F$ for some
$v'\in\bL_\uni^F$ implies that $v,v'$ are $\bL^F$-conjugate. It is clear that
there are exactly $|C_\bL(v)^F|$ many cosets $g\bU^F$ with $g^{-1}ug\in v\bU^F$.
\end{proof}

All of our examples indicate that the previous result continues to hold in the
general case (see also Corollary~\ref{cor:reg unip} and the tables in
Section~\ref{sec:Spin8}):

\begin{conj}   \label{conj:integral}
 The conclusions of Proposition~\ref{prop:split case} continue to hold for
 arbitrary $F$-stable Levi subgroups of $\bG$.
\end{conj}

Recall that a conjugacy class of $\bG^F$ is called \textit{uniform} if its characteristic
function is a linear combination of Deligne--Lusztig characters. Digne and
Michel \cite[Cor.~3.2]{DiMi0} (see also \cite[Cor.~4.4]{DiMi1}) have shown how
to compute the values of 2-parameter Green functions on uniform unipotent
conjugacy classes:

\begin{thm}   \label{thm:DM}
 Assume that either the class of $u\in\bG^F$ or of $v\in\bL^F$ is uniform. Then
 $$\QLG(u,v)=\frac{1}{|\bL^F|}\sum_{\bT\le\bL/\sim}
   \frac{|\bT^F|^2}{|N_\bL(\bT)^F|}\RTG(1)(u)\,\RTL(1)(v),$$
 where the sum runs over $\bL^F$-classes of $F$-stable maximal tori of $\bL$.
\end{thm}

\subsection{A linear system}
We write $\tQLG$ for the matrix $\big(|v^{\bL^F}|\QLG(u,v^{-1})\big)_{v,u}$,
with rows and columns indexed by the unipotent conjugacy classes of $\bL^F$,
$\bG^F$ respectively.

On unipotent elements $u\in\bG_\uni^F$ the character formula from
Theorem~\ref{thm:charform} reads
$$\RLG(\psi)(u)=\sum_{v\in\bL_\uni^F}\!\QLG(u,v^{-1})\,\psi(v)$$
for $\psi\in\Irr(\bL^F)$.

\begin{lem}   \label{lem:prod}
 Let $\bM\le\bG$ be an $F$-stable Levi subgroup containing $\bL$. Then
 $$\tQLG=\tQ_\bL^\bM\cdot\tQ_\bM^\bG\,.$$
\end{lem}

\begin{proof}
By transitivity of Lusztig induction \cite[Prop.~9.1.8]{DiMi2} we have
$\RLG=R_\bM^\bG\circ R_\bL^\bM$, so
$$\RLG(\psi)(u)=\sum_{v\in\bM_\uni^F}\!Q_\bM^\bG(u,v^{-1})\,
  \sum_{x\in\bL_\uni^F}\!Q_\bL^\bM(v,x^{-1})\,\psi(x)$$
for all $u\in\bG_\uni^F$ and all class functions $\psi$ on $\bL^F$. The claim
follows.
\end{proof}

Thus, for an inductive determination of the 2-parameter Green functions it is
sufficient to consider the case when $\bL<\bG$ is maximal among $F$-stable Levi
subgroups. Now let $\bT\le\bL$ be an $F$-stable maximal torus, then
Lemma~\ref{lem:prod} gives $\tRTG=\tRTL\cdot\tQLG$, where we have set
$\tRTG:=\big(\RTG(1)(u)\big)_u$.
The $\bL^F$-conjugacy classes of $F$-stable maximal tori of $\bL$ are
parametrised by $F$-conjugacy classes in the Weyl group $W_L$ of $\bL$ (see
\cite[4.2.22]{DiMi2}).
Thus the above formula yields a linear system of equations
$$R_{\bT_w}^\bG(1)(u)
  =\sum_{v\in\bL_\uni^F}R_{\bT_w}^{\bL}(1)(v)\,\QLG(u,v^{-1})\qquad(w\in W_L)
  \eqno{(*)}$$
for $\tQLG$ with coefficient matrix $\big(R_{\bT_w}^{\bL}(1)(v)\big)_{w,v}$.

\begin{prop}   \label{prop:solve}
 Assume that the number of unipotent classes of $\bL^F$ equals $|\Irr(W_L)^F|$.
 Then the 2-parameter Green functions $\QLG$ are uniquely determined by the
 ordinary Green functions of $\bG$ and of $\bL$.
\end{prop}

\begin{proof}
This follows from the above considerations and the fact that the square (by
assumption) matrix of values of Green functions on unipotent classes of $\bL^F$
is invertible due to the orthogonality relations for Green functions.
\end{proof}

The assumption of Proposition~\ref{prop:solve} is satisfied, for example, if
$\bL$ has connected centre and only components of type $A$, a case already
covered by Theorem~\ref{thm:DM}. Some examples will
be given in Section~\ref{subsec:GL_n}. For Levi subgroups of other types, the
above just gives restrictions on the values of the 2-parameter Green functions.
See also Section~\ref{sec:Spin8} below.

We then get the following values on regular unipotent elements of $\bG$:

\begin{cor}   \label{cor:reg unip}
 In the situation of Proposition~\ref{prop:solve}, if $u$ is regular unipotent
 in $\bG^F$ then
 $$\tQLG(u,v)=\begin{cases} 1& \text{for $v$ regular unipotent,}\\
   0& \text{otherwise.}\end{cases}$$
\end{cor}

\begin{proof}
It is known that the ordinary Green functions take constant value~1 on regular
unipotent elements (see \cite[Lemma 12.4.8]{DiMi2}). Clearly, the claimed values
yield a solution to the system of equations~($*$) for $\tQLG$.
\end{proof}

The conclusion needs no longer be true when $Z(\bL)$ is disconnected, see e.g. 
Proposition~\ref{prop:A3}.

\subsection{Some examples in $\GL_n$}   \label{subsec:GL_n}
We compute some examples of 2-parameter Green functions for $\GL_n$. For this
recall that the conjugacy classes of unipotent elements of $\GL_n$ and of
$\GL_n(q)$ are parametrised by partitions of $n$ via their Jordan canonical
form. We write $u_\la$ for a unipotent element in $\GL_n(q)$ with Jordan form of
shape $\la\vdash n$.

In $\GL_n$, the order relation on unipotent classes is given by the dominance
order on partitions. Proposition~\ref{prop:split case} then gives for example
the following vanishing result:
let $\bL$ be of type $\GL_{n-1}$ in $\bG$ of type $\GL_n$. Let
$u\in\bG_\uni$ and $v\in\bL_\uni$ be parametrised by $\la\vdash n$,
$\mu\vdash n-1$ respectively. Then $\QLG(u,v)=0$ unless $\mu$ is obtained from
$\la$ by removing one node. Indeed, the induced class is labelled by the
partition obtained by increasing the first part by~1, whence the claim follows
from Proposition~\ref{prop:split case}(b).

\begin{exmp}   \label{exmp:GLn}
(a) Let $\bG=\GL_3$ with the standard Frobenius map. Via their Jordan normal
form, the unipotent conjugacy classes of $\bG$ are parametrised by partitions
of~3, which we order $1^3,21,3$. For $\bL$ the only non-trivial standard Levi
subgroup, of type $A_1$, we find
$$\tQLG=\begin{pmatrix} \Ph3& 1& .\cr  .& q& 1
  \end{pmatrix}$$
where the columns are labelled by the unipotent classes of $\bG$, and the rows
by those of~$\bL$. Here, as in subsequent tables, $\Phi_d$ denotes the $d$th
cyclotomic polynomial evaluated at~$q$, and ``." stands for ``0".   \par
(b) Next let $\bG=\GL_4$, with the unipotent classes ordered as
$1^4,21^2,2^2,31,4$. First let $\bL_1$ be a standard Levi subgroup of type
$A_2$. Its unipotent classes
are labelled by the partitions $1^3,21,3$ of $3$, and we obtain the matrix of
(modified) 2-parameter Green functions
$$\tQ_{\bL_1}^\bG=\begin{pmatrix}
  \Ph2\Ph4& 1& .& .& .\cr
  .& q\Ph2& \Ph2& 1& .\cr
  .& .& .& q& 1
\end{pmatrix}.$$
Next, let $\bL_2$ be the standard Levi subgroup of type $A_1^2$. The
resulting matrix is
$$\tQ_{\bL_2}^\bG=\begin{pmatrix}
  \Ph3\Ph4& \Ph2& 1& .& .\cr
  .& q^2& .& 1& .\cr
  .& q^2& .& 1& .\cr
  .& .& q\Ph2& \Ph1& 1
\end{pmatrix}$$
where the rows are labelled by the unipotent conjugacy classes of $\bL_2$
parametrised by the pairs $(1^2,1^2),(2,1^2),(1^2,2),(2,2)$ of partitions of~2.
Note that the second and third row agree, as the second and third unipotent
class of $\bL_2^F$ are conjugate in $N_\bG(\bL)^F$. For the twisted Levi
subgroup $\bL_3$ of type $A_1(q^2).(q^2-1)$ we find
$$\tQ_{\bL_3}^\bG=\begin{pmatrix}
  \Ph1^2\Ph3& -\Ph1& 1& .& .\cr
  .& .& q\Ph1& -\Ph1& 1
\end{pmatrix}.$$
\par
Finally, the split Levi subgroup of type $A_1$ is not maximal and we may apply
Lemma~\ref{lem:prod}, while for its twisted version $\bL_4$ of type
$A_1(q).(q^2-1)(q-1)$ we obtain
$$\tQ_{\bL_4}^\bG=\begin{pmatrix}
  -\Ph1\Ph3\Ph4& 1& -\Ph1& 1& .\cr
  .& -q^2\Ph1& q\Ph2& .& 1
\end{pmatrix}.$$
\par
(c) For $\bG=\GL_5$ and a standard Levi subgroup $\bL_1$ of type $A_3$ we obtain
$$\tQ_{\bL_1}^\bG=\begin{pmatrix}
  \Ph5& 1& .& .& .& .& .\cr
  .& q\Ph3& \Ph2& 1& .& .& .\cr
  .& .& q^2& .& 1& .& .\cr
  .& .& .& q\Ph2& q& 1& .\cr
  .& .& .& .& .& q& 1
\end{pmatrix},$$
while for a standard Levi subgroup $\bL_2$ of type $A_2A_1$ we find
$$\tQ_{\bL_2}^\bG=\begin{pmatrix}
  \Ph4\Ph5& \Ph3& 1& .& .& .& .\cr
  .& q^3& .& 1& .& .& .\cr
  .& q^2\Ph3& q\Ph2& \Ph2& 1& .& .\cr
  .& .& q^2\Ph2& \Ph1\Ph2& q& 1& .\cr
  .& .& .& q^2& .& 1& .\cr
  .& .& .& .& q^2& \Ph1&  1\cr
\end{pmatrix}$$
where the classes of $\bL$ are ordered as $(1^3,1^2),(1^3,2),(21,1^2),(21,2),
(3,1^2),(3,2)$.
Note that here two distinct unipotent classes of $\bL$ fuse into the same
$\bG$-class (corresponding to rows~2 and~3) but the Green functions differ: the
classes are not $N_\bG(\bL)^F$-conjugate.
\end{exmp}

The above examples seem to indicate the following conjecture:

\begin{conj}   \label{conj:type A}
 Let $\bG$ be of type $A$, $\bL\le\bG$ an $F$-stable Levi subgroup,
 $v\in\bL_\uni^F$ and $u$ is the induced class of $v^\bL$ in the sense of
 Lusztig--Spaltenstein. Then $\tQLG(u,v)=1$.
\end{conj}

For example, if $\bL=\GL_m.\GL_{n-m}<\bG=\GL_n$ and $v$ is parametrised by
$(\la_1,\la_2)$ then we should have $\tQLG(u,v)=1$ for $u$ parametrised by
$\la_1+\la_2$.

\subsection{Some examples in classical groups}
In this section we consider some further cases for $\bG$ with connected centre
and $F$ a Frobenius endomorphism with respect to an $\FF_q$-structure, where
$q$ is odd. In this situation, the assumptions of Proposition~\ref{prop:solve}
are satisfied for groups of adjoint types $B_2,B_3,C_3,D_4$, for example. 

\begin{exmp}   \label{exmp:class}
(a) Let $\bG=\SO_5$ and $\bL$ a Levi subgroup of type $A_1$ containing long root
subgroups. Then we find
$$\begin{pmatrix}
  \Ph2\Ph4& 1& 2& .& .\cr
  .& q\Ph2& \Ph1& \Ph2& 1
\end{pmatrix}\quad\text{and}\quad
\begin{pmatrix}
  -\Ph1\Ph4& 1& .& 2& .\cr
  .& q\Ph1& -\Ph1& -\Ph2& 1
\end{pmatrix}$$
for the split and twisted version, respectively, while for the short root $A_1$s
we get
$$\begin{pmatrix}
  \Ph2\Ph4& \Ph2& \Ph2& 1& .\cr
  .& .& 2q& .& 1
\end{pmatrix}\quad\text{and}\quad
\begin{pmatrix}
  -\Ph1\Ph4& -\Ph1& 1& 1& .\cr
  .& .& .& -2q& 1
\end{pmatrix}.$$
\par
(b) Let $\bG=\SO_7$ and $\bL$ be a split Levi subgroup of type $B_2$. Here,
$\bL$ has again as many classes as its Weyl group has irreducible characters,
so the system has a unique solution
$$\tQLG=\begin{pmatrix}
  \Ph2\Ph3\Ph6& \Ph2& 1& 1& .& .& .& .& .& .\cr
  .& q^2\Ph2& .& .& 1& 2& .& .& .& .\cr
  .& .& q\Ph2^2& .& q\Ph2& \hlf q\Ph1& \hlf q\Ph2& 1& .& .\cr
  .& .& .& q\Ph4& .& \hlf q\Ph1& \hlf q\Ph2& .& 1& .\cr
  .& .& .& .& .& .& .& 2q& .& 1
\end{pmatrix}.$$
\par
(c) Let $\bG$ be of adjoint type $C_3$ and $\bL$ split of type $C_2$. Then we
find
$$\tQLG=\begin{pmatrix}
  \Ph2\Ph3\Ph6& 1& 2& .& .& .& .& .& .& .\cr
  .& q\Ph2\Ph4& \Ph1& \Ph2& \Ph2& .& 1& .& .& .\cr
  .& .& q^2\Ph2& .& \hlf q\Ph2& \Ph2& .& 1& .& .\cr
  .& .& .& q^2\Ph2& \hlf q\Ph1& .& .& .& 1& .\cr
  .& .& .& .& .& .& q\Ph2& q& q& 1
\end{pmatrix}.$$
\par
(d) Let $\bG$ be of type $D_4$ with connected centre, $F$ a split Frobenius and
$\bL$ a split Levi subgroup of type $A_3$. Ordering the 13 unipotent classes of
$\bG^F$ by their Jordan normal forms
$$1^8,\ 2^21^4,\ 2^4+,\ 2^4-,\ 31^5,\ 32^21,\ 3^21^2\ (\text{two classes}),\ 
  4^2+,\ 4^2-,\ 51^3,\ 53,\ 71$$
we find
$$\left(\begin{array}{ccccccccccccc}
   \Ph2^2\Ph4\Ph6& \Ph2& .& .& 1& .& .& .& .& .& .& .& .\\
   .& q^2\Ph2^2& \Ph2\Ph4& \Ph2\Ph4& .& 1& 2& .& .& .& .& .& .\\
   .& .& .& .& q\Ph2\Ph4& q\Ph2& \Ph1& \Ph2& 1& .& .& .& .\\
   .& .& .& .& .& .& 2q^2 & .& .& \Ph2& \Ph2& 1& .\\
   .& .& .& .& .& .& .& .& q\Ph2& .& .& q& 1
\end{array}\right)$$
for one of the three possible embeddings, and a suitable permutation of the
columns $(3,4,5)(9,10,11)$ for the other two. The Green functions for a twisted
Levi subgroup of type $\tw2A_3(q).(q+1)$ are related to those of $A_3(q).(q-1)$
as follows: The ordinary Green functions of $\tw2A_3(q)$ are obtained from those
of $A_3(q)$ by replacing $q$ with $-q$ by Ennola duality \cite{Kaw85}. This also
entails a permutation of maximal tori in the linear system~($*$). Therefore,
the Green functions
of $\tw2A_3(q).(q+1)$ are obtained from those of $A_3(q).(q-1)$ by replacing
$q$ by $-q$ and swapping the classes with Jordan normal form $3^21^2$ (which
have centraliser orders $q^8(q\pm 1)^2$). \par
Next consider a split Levi subgroup $\bL$ of type $A_1^3$. Here, $\tQLG$ equals
{\scriptsize $$\left(\begin{array}{ccccccccccccc}
  \Ph2\Ph3\Ph4^2\Ph6& *& \Ph2\Ph4& \Ph2\Ph4& \Ph2\Ph4& \Ph2& 1& 1& .& .& .& .& .\cr
  .& q^4\Ph2& .& .& q^2\Ph4& .& 2q& .& .& .& 1& .& .\cr
  .& q^4\Ph2& .& q^2\Ph4& .& .& 2q& .& .& 1& .& .& .\cr
  .& q^4\Ph2& q^2\Ph4& .& .& .& 2q& .& 1& .& .& .& .\cr
  .& .& \!\!q^2\Ph2\Ph4& .& .& q^2& 2q\Ph1& .& \Ph2& .& .& 1& .\cr
  .& .& .& \!\!q^2\Ph2\Ph4& .& q^2& 2q\Ph1& .& .& \Ph2& .& 1& .\cr
  .& .& .& .& \!\!q^2\Ph2\Ph4& q^2& 2q\Ph1& .& .& .& \Ph2& 1& .\cr
  .& .& .& .& .& \!\!q^3\Ph2& q\Ph1^2& q\Ph2^2& \Ph1\Ph2& \Ph1\Ph2& \Ph1\Ph2& q\mn2& 1
\end{array}\right)$$}
where $*=q^4+3q^3+3q^2+ q+1$. The symmetry from triality,
cyclically permuting the classes 3,4,5 and 9,10,11 of $\bG^F$, as well as the
classes 2,3,4 and 5,6,7 of $\bL^F$, is clearly visible. Again, the Green
functions for a twisted Levi subgroup of type $A_1(q)^3.(q+1)$ are obtained
by replacing $q$ with $-q$ and interchanging the two classes with Jordan normal
form $3^21^2$. For a split Levi subgroup of type $A_1(q^2).(q^2+1)$ we find
{\scriptsize $$\left(\begin{array}{ccccccccccccc}
-\Phi_1^3\Phi_2^3\Phi_3\Phi_6 & \Phi_1^2\Phi_2^2 & \Phi_1^2\Phi_2^2 & \Phi_1^2\Phi_2^2 & -\Phi_1\Phi_2 & -\Phi_1\Phi_2 & -\Phi_1 & \Phi_2 & . & . & 1 & . & . \\
. & . & . & . & -q^2\Phi_1^2\Phi_2^2 & q^2\Phi_1\Phi_2 & q\Phi_1^2 & -q\Phi_2^2 & -\Phi_1\Phi_2 & -\Phi_1\Phi_2 & . & 1 & 1
\end{array}\right)$$}
\par
(e) Let $\bG$ be of type $D_4$ with twisted Frobenius such that
$\bG^F=\tw2D_4(q)$ and $\bL$ an $F$-stable Levi subgroup of twisted type
$\tw2A_3(q).(q-1)$. Then we get
$$\tQLG=\begin{pmatrix}
  \Ph3\Ph8& \Ph2& 1& .& .& .& .& .& .\cr
  .& q^2\Ph4& .& 1& .& 2& .& .& .\cr
  .& .& q\Ph2\Ph4& q\Ph2& \Ph2& \Ph1& 1& .& .\cr
  .& .& .& .& 2q^2& .& .& 1& .\cr
  .& .& .& .& .& .& q\Ph2& q& 1
\end{pmatrix}.$$
\par
(f) Let $\bG$ be of type $D_4$ with triality Frobenius and $\bL$ an $F$-stable
Levi subgroup of type $A_1(q^3).(q-1)$. Then we obtain
$$\tQLG=\begin{pmatrix}
  \Ph2\Ph3\Ph6\Ph{12}& q^4\pl q\pl1& \Ph2& 1& 1& .& .\cr
  .& .& q^3\Ph2& q\Ph3& q\Ph6& \Ph2& 1
\end{pmatrix},$$
while for a Levi subgroup of type $A_1(q).(q^3-1)$ one has
$$\tQLG=\begin{pmatrix}
  \Ph2\Ph3\Ph6^2\Ph{12}& \Ph2\Ph6& q^3\pl q^2\pl1& \Ph3& \Ph6& 1& .\cr
  .& q^4\Ph2\Ph6& q^2\Ph1\Ph2& q(q^2\mn q\mn1)& q\Ph6& .& 1
\end{pmatrix}.$$
\end{exmp}

\section{The Green functions for $\Spin_8^+(q)$}   \label{sec:Spin8}
In this section we determine all 2-parameter Green functions from maximal
$F$-stable Levi subgroups of the simply connected groups of type~$D_4$ with a
split Frobenius, that is, for the groups $\Spin_8^+(q)$ with $q$ an odd prime
power. For this, we first study the relation to the Green functions for groups
with connected centre.

\subsection{Groups with non-connected centre}
For groups $\bG$ with non-connected centre, the number of unipotent classes is
bigger than the number of irreducible characters of the Weyl group, so that
Proposition~\ref{prop:solve} cannot be applied.
Still, additional considerations sometimes lead to a solution.
\par
For this, let $\bG\hookrightarrow\tbG$ be an $F$-equivariant regular embedding,
so that $\tbG=\bG Z(\tbG)$ has connected centre and derived subgroup
$[\tbG,\tbG]=\bG$. Then for any Levi subgroup $\bL\le\bG$, $\tbL:=\bL Z(\tbG)$
is a Levi subgroup of $\tbG$, $F$-stable if $\bL$ is. As in our case, when
$\bG$ is of type $D_4$, all proper Levi subgroups of $\bG$, and hence of $\tbG$,
are of type $A$, all 2-parameter Green functions for $\tbG$ can be computed
with Proposition~\ref{prop:solve}. Then we just need to descend to the simply
connected group $\bG$.
\par
Let us write $\cL_\bG$ for the Lang map on $\bG$. Now while $\cL_\bG^{-1}(\bU)$
is not necessarily invariant under the action of $\tbG\times\tbL$, it is so
under the diagonally embedded subgroup $\Delta(\tbL^F)\cong\tbL^F$ of
$\tbG^F\times\tbL^F$ (with $\Delta(l)=(l,l^{-1})$ for $l\in \tbL^F$).
So $\Delta(\tbL^F)$ acts on the $\ell$-adic cohomology
groups $\text{H}_c^i(\cL_\bG^{-1}(\bU))$. In particular, the 2-parameter Green
functions $\QLG$ are invariant under the diagonal action of $\tbL^F$:
$$\QLG(u,v)=\QLG(u^l,v^l)\qquad\text{for all }l\in\tbL^F.$$
Furthermore, the 2-parameter Green functions $Q_\tbL^\tbG$ are induced
from those of $\bL$ inside $\bG$ (see \cite[Prop.~2.2.1(b)]{Bo00}).
This implies relations between the 2-parameter Green functions of $\tbG$
and $\bG$. Suppose that the unipotent class of $u\in \tbG^F$ splits into $n$
classes of $\bG^F$ and that the unipotent class of $v\in \tbL^F $ splits into
$m$ classes of $\bL^F$, then by \cite[Prop.~2.2.1(b)]{Bo00}
\begin{equation}   \label{eq3}
\tQ_\tbL^\tbG(u,v) = \sum_{i=1}^m \tQLG(u,v_i)= \frac{m}{n}\sum_{i=1}^n \tQLG(u_i,v)
\end{equation}
where $u_i\in \bG^F$ are the representatives of the $\bG^F$-classes in
$\left[u\right]_{\tbG^F}$, and $v_i\in \bL^F$ are the representatives of the
$\bL^F$-classes in $\left[u\right]_{\tbL^F}$.

In the examples that we consider below, five cases appear:
\begin{itemize}
\item If $n=1$ then (\ref{eq3}) and the $\tilde{L}^F$-invariance directly gives
 $$\tQLG(u,v_i)=\frac{1}{m}\tQ_\tbL^\tbG(u,v)\qquad\text{for } i=1,\ldots,m.$$
\item If $m=1$ then (\ref{eq3}) and the $\tilde{L}^F$-invariance directly gives
 $$\tQLG(u_i,v)=\tQ_\tbL^\tbG(u,v)\qquad\text{for }i=1,\ldots,n. $$
\item If $n=m=2$ then the value $Q:=\tQ_\tbL^\tbG(u,v)$ is replaced in the
 matrix of $\tQLG$ by the submatrix
$$\begin{array}{c|cc}
 & u_1 & u_2 \\ \hline
v_1 & a & Q-a \\
v_2 & Q-a & a 
\end{array}$$
 where $a:=\tQLG(u_1,v_1)$.
\item If $n=4$, $m=2$ then the value $Q:=\tQ_\tbL^\tbG(u,v)$ is replaced in the
 matrix of $\tQLG$ by the submatrix
$$\begin{array}{c|cccc}
 & u_1 & u_2 & u_3 & u_4 \\ \hline
v_1 & a_1 & a_2 & a_3 & 2Q-a_1-a_2-a_3 \\
v_2 & Q-a_1 & Q-a_2 & Q-a_3 & a_1+a_2+a_3-Q 
\end{array}$$
 where $a_i:=\tQLG(u_i,v_1)$ for $i=1,2,3$.
\item If $n=2$, $m=4$ then the value $Q:=\tQ_\tbL^\tbG(u,v)$ is replaced in the
 matrix of $\tQLG$ by the submatrix
$$\begin{array}{c|cc}
 & u_1 & u_2 \\ \hline
v_1 & a_1 & Q/2-a_1 \\
v_2 & a_2 & Q/2-a_2 \\
v_3 & a_3 & Q/2-a_3 \\
v_4 & Q-a_1-a_2-a_3 & a_1+a_2+a_3-Q/2 
\end{array}$$
 where $a_i:=\tQLG(u_1,v_i)$ for $i=1,2,3$.
\end{itemize}

So, every pair of splitting classes introduces some unknown entries in $\tQLG$.
The number of unknowns can be reduced thanks to the diagonal action of $\tbL^F$.

We get the explicit action of $\tbL^F$ by following the construction of
Bonnaf\'e in \cite{Bo00} around Proposition 2.2.2. There, he translates the
action of an element $l\in\tbL^F/\bL^F$ (on the unipotent elements of $\bG^F$
and $\bL^F$) to the action (by conjugation) of a representative $l_z\in\bL$
for $z \in H^1(F,Z(\bL))$ such that $l_z^{-1}F(l_z)=z$. Since there is a
canonical surjection $H^1(F,\bZ) \twoheadrightarrow H^1(F,Z(\bL))$ we
choose, instead, for all $z\in H^1(F,\bZ)$ an element $g\in\bG$ such that
$g^{-1}F(g)=z$, where $\bZ=Z(\bG)$. This gives us the desired action for all
split Levi subgroups. For the non-split ones we replace $F$ by the associated
twisted Frobenius map $F'$. So for all $z \in H^1(F,\bZ)$ we have
$g,g'\in\bG$ such that $g^{-1}F(g)=z=g'^{-1}F'(g')$. Then the action is given
by $(u,v)\mapsto (u^g,v^{g'})$ for $u \in \bG^F$ and $v\in \bL^{F'}$.

We will see how to find the remaining unknowns in the following examples.

Recall that a class function $f$ of $\bG^F$ is \emph{cuspidal} if for every
$F$-stable proper Levi subgroup $\bL$ of $\bG$ we have $\tw{*}\RLG f=0$.

\begin{rem}
For any Levi subgroup $\bL$ of $\bG$ minimal with respect to having disconnected
centre a class function $f$ of $\bL^F$ that takes 1, -1 on a splitting
unipotent class (say with
representatives $u_1$ and $u_2$) and zero elsewhere is cuspidal.
By the discussion above and minimality, for any proper Levi subgroup $\bL'$ of
$\bL$ we have that $\tQ_{\bL'}^\bL(u_1,v)=\tQ_{\bL'}^\bL(u_2,v)$ for all
$v \in {\bL'}^F_\uni$. Then $\tw{*}R_{\bL'}^\bL(f)=0$.
\end{rem}

\subsection{The Green functions for $\SL_4(q)$ and $\SU_4(q)$}   \label{subsec:SL4}
It will be convenient (in view of Lemma~\ref{lem:prod}) to first compute the
Green functions for Levi subgroups of types $A_3$ and~$\tw2A_3$.

\begin{prop}   \label{prop:A3}
 Let $q$ be an odd prime power, $\bG=\SL_4$ when $q\equiv3\pmod4$, or
 $\bG=\SL_4/\langle\pm1\rangle$ when $q\equiv1\pmod4$, and $F$ a split Frobenius
 with respect to an $\FF_q$-structure on $\bG$. Then
 $$\tQ_{\bL_1}^\bG=\begin{pmatrix}
  \Ph3\Ph4& \Ph2& 1& 1& .& .& .\cr
  .& q^2& .& .& 1& .& .\cr
  .& q^2& .& .& 1& .& .\cr
  .& .& q\Ph2& .& \frac{1}{2}\Ph1& 1& .\cr
  .& .& .& q\Ph2& \frac{1}{2}\Ph1& .& 1\cr
 \end{pmatrix}$$
 for a split Levi subgroup $\bL_1$ with $\bL_1^F=A_1(q)^2.(q-1)$, and
 $$\tQ_{\bL_2}^\bG=\begin{pmatrix}
  \Ph1^2\Ph3& -\Ph1& 1& 1& .& .& .\cr
  .& .& q\Ph1& .& -\frac{1}{2}\Ph1& .& 1\cr
  .& .& .& q\Ph1& -\frac{1}{2}\Ph1& 1& .
 \end{pmatrix}$$
 for a non-split Levi subgroup $\bL_2$ with $\bL_2^F=A_1(q^2).(q+1)$.
\end{prop}

\begin{proof}
In both cases, $G=\bG^F$ has centre of order~2 and seven unipotent classes with
representatives $u_1,\ldots,u_7$, since the classes of $\bG$ with Jordan normal
forms $2^2$ and $4$ both split into two classes in the finite group. Moreover,
as $\bL_1$ is split of type $A_1^2$, the Levi subgroup $\bL_1^F$ has five
unipotent classes with representatives $v_1,\ldots,v_5$ of Jordan types
$(1^2,1^2),(1^2,2),(2,1^2)$ and two of type $(2,2)$. Now by the known Green
functions for $\tbL_1$ in $\tbG$ (see Example~\ref{exmp:GLn}(b)) and our
considerations above, we only need to worry about $u$ and $v$ having two blocks
of size 2 or $u$ being regular unipotent. From the explicit realisation in the
split case as value of a permutation character (see~(\ref{eq2})) one computes
that $Q_{\bL_1}^\bG(u,v^{-1})=0$ for $u,v$ with two Jordan blocks, but not
conjugate in $\bG^F$. Also, $Q_{\bL_1}^\bG(u,v^{-1})=0$ for $v$ with two Jordan
blocks, and $u$ one out of the two regular unipotent classes. So, with respect
to a suitable ordering of the splitting classes, we find the stated result. 
\par
Now assume that $\bL_2$ is of non-split type $A_1^2$, so
$\bL_2^F=A_1(q^2).(q+1)$. The class of regular unipotent elements of $\bL_2$
splits into two $\bL_2^F$-classes, with representatives $v_2,v_3$ say. Again,
by the known Green functions for $\tbL_2$ in $\tbG$ (see
Example~\ref{exmp:GLn}(b)) and our considerations above there are two unknown
values in $\tQ_{\bL_2}^\bG$, denoted $a_1:=\tQ_{\bL_2}^\bG(u_3,v_2)$ and
$a_2:=\tQ_{\bL_2}^\bG(u_6,v_2)$.
\par
Consider the cuspidal class function $\psi$ on $\bL_2^F$ that takes values 1 and
$-1$ on $v_2$, $v_3$ respectively, and zero everywhere else. By the Mackey formula
this means, since $|N_\bG(\bL_2)^F:\bL_2^F|=2$, that
$\langle R_{\bL_2}^\bG(\psi),R_{\bL_2}^\bG(\psi)\rangle
=N\langle\psi,\psi\rangle$, with $N$ being equal 1 or 2 (depending on whether
$v_2$ and $v_3$ fuse in $\bG^F$). On the other hand, the norm of
$R_{\bL_2}^\bG(\psi)$ can be computed from the character formula using the
(unknown) values of $\tQ_{\bL_2}^\bG(u,v)$ on the pairs of splitting classes.
Then the norm equation above gives that
$$(2a_1-q\Ph1)^2+q^2\Ph1\Ph2(2a_2-1)^2=Nq^3\Ph1.$$
Next, let $\xi$ be the class function on $\bL_1^F$ that takes values 1
resp.~$-1$ on the two regular unipotent classes, respectively. Since $\psi$ is
cuspidal, the Mackey formula shows that
$\langle R_{\bL_1}^\bG(\xi),R_{\bL_2}^\bG(\psi)\rangle=0$. Using the known
values of $\tQ_{\bL_1}^\bG$, this translates to $a_1=q\Ph1(1-a_2)$. This system
of equations has rational solutions only for $N=2$ (meaning that the classes of
$v_2$ and $v_3$ don't fuse in $\bG^F$). The solutions of this system are
$(a_1,a_2)\in\left\lbrace (q\Ph1,0),\, (0,1) \right\rbrace $. 
Both correspond to the a matrix as in the statement, but in one case the
second and third lines are interchanged. This was to be expected, since we did
not choose explicit representatives for the classes of $v_2$ and $v_3$.
\end{proof}

To fix the table for a given choice of representative, we can proceed as
follows. Both solutions verify Conjecture~5.2 of \cite{DLM92} which in turn
proves Conjecture~5.2' of the same article in our case. This states that the
Lusztig restriction of Gel'fand--Graev characters of $\bG^F$ are Gel'fand--Graev
characters of $\bL_2^F$ (up to a possible sign, easily chosen by imposing the
positivity of the degree). By explicit computation of the Gel'fand--Graev
characters this eliminates one of the solutions and thus leads to a unique
result.

This result makes it easy to prove the following:

\begin{prop}   \label{prop:SL4,q=1(4)}
 Let $\bG=\SL_4$ and $\bG^F=\SL_4(q)$ with $q\equiv1\pmod4$. Then
 $$\tQ_{\bL_1}^\bG=\begin{pmatrix}
   \Ph3\Ph4& \Ph2& 1& 1& .& .& .& .& .\cr
   .& q^2& .& .& 1& .& .& .& .\cr
   .& q^2& .& .& 1& .& .& .& .\cr
   .& .& q\Ph2& .& \frac{1}{2}\Ph1& 1& .& 1& .\cr
   .& .& .& q\Ph2& \frac{1}{2}\Ph1& .& 1& .& 1
 \end{pmatrix}$$
 for a split Levi subgroup $\bL_1$ of type $A_1^2$, and
 $$\tQ_{\bL_2}^\bG=\begin{pmatrix}
   \Ph1^2\Ph3& -\Ph1& 1& 1& .& .& .& .& .\cr
  .& .& q\Ph1& .& -\frac{1}{2}\Ph1& .& 1& .& 1\cr
  .& .& .& q\Ph1& -\frac{1}{2}\Ph1& 1& .& 1& .
 \end{pmatrix}$$
 for a non-split Levi subgroup $\bL_2$ of type $A_1^2$.
\end{prop}

\begin{proof}
This directly follows from Proposition~\ref{prop:A3} and 
\cite[Cor.~2.2.3]{Bo00} applied to the surjection
$\varphi:\SL_4\to\SL_4/\langle\pm1\rangle$ with $\ker(\varphi)^\circ=1$.
We have checked the condition for the latter explicitly (by replacing the
Frobenius morphism with the twisted Frobenius associated with the twisted Levi).
\end{proof}

Here is the counter-part of Proposition~\ref{prop:A3} for $\SU_4(q)$:

\begin{prop}   \label{prop:2A3}
 Let $q$ be an odd prime power and $\bG=\SL_4$ when $q\equiv1\pmod4$, or
 $\bG=\SL_4/\langle\pm1\rangle$ when $q\equiv3\pmod4$, and $F$ a twisted
 Frobenius with respect to an $\FF_q$-structure on $\bG$. Then
 $$\tQ_{\bL_1}^\bG=\begin{pmatrix}
  \Ph2^2\Ph6& \Ph2& 1& 1& .& .& .\cr
  .& .& q\Ph2& .& \hlf\Ph2& 1& .\cr
  .& .& .& q\Ph2& \hlf\Ph2& .& 1
 \end{pmatrix}$$
 for a split Levi subgroup $\bL_1$ with $\bL_1^F=A_1(q^2).(q-1)$, and
 $$\tQ_{\bL_2}^\bG=\begin{pmatrix}
  \Ph4\Ph6& -\Ph1& 1& 1& .& .& .\cr
  .& q^2& .& .& 1& .& .\cr
  .& q^2& .& .& 1& .& .\cr
  .& .& q\Ph1& .& -\hlf\Ph2& .& 1\cr
  .& .& .& q\Ph1& -\hlf\Ph2& 1& .\cr
 \end{pmatrix}$$
 for a non-split Levi subgroup $\bL_2$ with $\bL_2^F=A_1(q)^2.(q+1)$.
\end{prop}

\begin{proof}
Instead of doing the same computation twice for the stated groups and
congruences
we did the computation once for the Levi subgroups $A_1(q^2).(q^2-1)$ and 
$A_1(q)^2.(q+1)^2$ inside the Levi $\tw2A_3(q)(q+1)$ of $\Spin_8^+(q)$.
Then, \cite[Cor.~2.2.3]{Bo00} assures that the solution is what we claim.
The argument is analogous to the proof of Proposition~\ref{prop:A3}.
By explicit calculations one finds the result for the split case. Then, thanks
to scalar products of Lusztig-induced cuspidal class functions and the norm
equation one can fix the unknowns for the non-split case.
\end{proof}

Again the Lusztig restriction of Gel'fand--Graev characters can be used to fix
the position of the values.

\subsection{The Green functions for $\Spin_8^+(q)$}   \label{subsec:Spin8+}
Let $\bG=\Spin_8$ with a split Frobenius $F$, such that $\bG^F=\Spin_8^+(q)$
with $q$ odd. Moreover, let $\bG\rightarrow\tbG$ be a regular embedding.
Of the 13 unipotent classes of $\bG$ six split into two and three split into
four classes in $\bG^F$. By the general discussion at the beginning of the
section this means that every splitting unipotent class of a Levi subgroup $\bL$
introduces at least 9 unknowns in the matrix $\tQLG$. It turns out that this
number can be reduced to 5 thanks to the diagonal action of $\tbL^F$. So we are
missing 5 equations per Levi subgroup and splitting class of that Levi.

The types of maximal Levi subgroups with disconnected centre of $\Spin_8^+(q)$
are $M_1=A_3(q)(q-1)$,  $M_2=A_1(q)^3(q-1)$, $M_3=\tw2A_3(q)(q+1)$,
$M_4=A_1(q)^3(q+1)$ and $L_5=A_1(q^2)(q^2+1)$. It turns out that to compute the
Green functions for these Levis we need also to consider minimal Levi
subgroups with disconnected centre $L_1=A_1(q)^2(q-1)^2$,
$L_2=A_1(q^2)(q^2-1)$, $L_3=A_1(q)^2(q^2-1)$ and $L_4=A_1(q)^2(q+1)^2$.
Figure~\ref{fig:Levi_Lattice} summarises the situation by showing the diagram of
inclusions for these Levi subgroups. Each of these has one splitting
unipotent class, and we denote by $f_i$ the (cuspidal) class function of $L_i$
that has values 1, -1 on the splitting classes and 0 elsewhere.

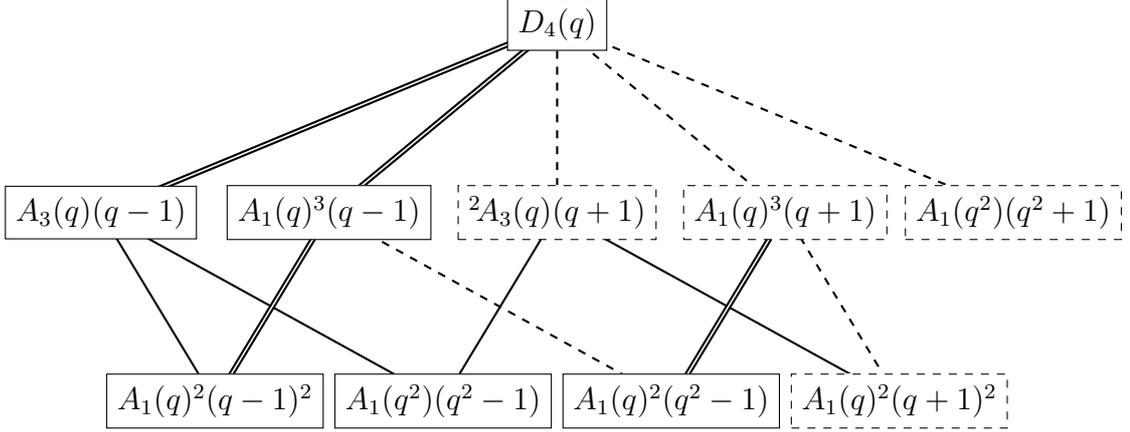
\begin{figure}[ht]
\caption{Subgroup lattice of Levi subgroups  with disconnected centre of
 $\Spin_8^+(q)$, up to triality. The lines represent inclusions: a single line
 corresponds to Levis already treated in Propositions~\ref{prop:A3}
 and~\ref{prop:2A3}, a double line indicates a split Levi and, therefore, that
 the Green function can be computed explicitly using~(\ref{eq2}), dashed lines
 indicate non-split Levis for which the Green functions must be computed with
 methods similar to those used in the proof of Proposition~\ref{prop:A3}.}   \label{fig:Levi_Lattice}
\begin{tikzpicture}
\node(a) at (6,5) {};

\node(b1) at (0,2.5){};
\node(b2) at (3,2.5) {};
\node(b3) at (6,2.5) {};
\node(b4) at (9,2.5) {};
\node(b5) at (12,2.5) {};

\node(c1) at (1.5,0) {};
\node(c2) at (4.5,0) {};
\node(c3) at (7.5,0) {};
\node(c4) at (10.5,0) {};

\draw [thick,double] (a) -- (b1);
\draw [thick,double] (a) -- (b2);
\draw [thick,dashed] (a) -- (b3);
\draw [thick,dashed] (a) -- (b4);
\draw [thick,dashed] (a) -- (b5);

\draw [thick] (b1) -- (c1);
\draw [thick] (b1) -- (c2);

\draw [thick,double] (b2) -- (c1);
\draw [thick,dashed] (b2) -- (c3);

\draw [thick] (b3) -- (c4);
\draw [thick] (b3) -- (c2);

\draw [thick,double] (b4) -- (c3);
\draw [thick,dashed] (b4) -- (c4);

\node[fill=white,rectangle,draw] at (a) {$D_4(q)$};

\node[fill=white,rectangle,draw] at (b1) {$A_3(q)(q-1)$};
\node[fill=white,rectangle,draw] at (b2) {$A_1(q)^3(q-1)$};
\node[fill=white,rectangle,draw,dashed] at (b3) {$\tw2A_3(q)(q+1)$};
\node[fill=white,rectangle,draw,dashed] at (b4) {$A_1(q)^3(q+1)$};
\node[fill=white,rectangle,draw,dashed] at (b5) {$A_1(q^2)(q^2+1)$};

\node[fill=white,rectangle,draw] at (c1) {$A_1(q)^2(q-1)^2$};
\node[fill=white,rectangle,draw] at (c2) {$A_1(q^2)(q^2-1)$};
\node[fill=white,rectangle,draw] at (c3) {$A_1(q)^2(q^2-1)$};
\node[fill=white,rectangle,draw,dashed] at (c4) {$A_1(q)^2(q+1)^2$};
\end{tikzpicture}
\end{figure}

The reason why we need $L_1$ to $L_4$ is because we want to emulate the proof
of Proposition~\ref{prop:A3}. While 
$$\big\langle R_{\bL_i}^\bG f_i,\, R_{\bL_j}^\bG f_j \big\rangle = 0
  \qquad\text{for }i,j=1,\dots,5,\ i\neq j$$
by the Mackey formula, an analogous formula for $M_i$, $i=1,\dots,4$, need not
hold since a class function $f$ of $M_i$ taking 1, -1 on splitting
unipotent classes is not cuspidal. However, 
$$\big\langle R_{\bM_i}^\bG f,\,R_{\bL_j}^\bG f_j \big\rangle = 0
  \quad\text{for all $i,j$ such that $L_j\nsubseteq M_i$}$$
again, by the Mackey formula.

In conclusion, the plan to get the 2-parameter Green functions of
$\Spin_8^+(q)$ is the following. First, we determine $\tQLG$ for the split
Levi subgroups $\bL$ by explicit computations with (2). Second, we compute
$\tQ_{\bL_i}^\bG$ for $i=2,\dots,5$ thanks to Propositions~\ref{prop:A3}
and~ \ref{prop:2A3} and Lemma~\ref{lem:prod}. Then using this knowledge and
Lemma~\ref{lem:prod} we get $\tQLG$ for the maximal Levi subgroups $\bL$.

Unfortunately, for the non-split Levis only 4 of the 5 equations (one of which
is the norm equation) needed to find the unknowns can be obtained by the
procedure described above. However, Digne, Lehrer and Michel prove in
\cite[Thm.~3.7]{DLM97} and \cite[Conj.~5.2 and 5.2']{DLM92} that, when $q$ is
"large enough" ($q>q_0$ for $q_0$ a constant depending only on the Dynkin
diagram), for all regular unipotent elements $u$ of $\tbG^F$, $\tQLG(u,v)=1$
for exactly one regular unipotent class $(v)^{\bL^F}$ and $\tQLG(u,v')=0$ for
all other regular unipotent elements $v'$ of $\bL^F$. This gives the last
equation.

\begin{thm}   \label{thm:Spin8}
 The 2-parameter Green functions for $G=\Spin_8^+(q)$, $q$ odd, satisfy:
 \begin{enumerate}
  \item[\rm(a)] For the maximal $F$-stable split Levi subgroups $\bL$ of $\bG$,
   $\tQLG$ is given in Tables~\ref{tab:split_a3}--\ref{tab:split_a1a1a1},
  \item[\rm(b)] For the maximal $F$-stable non-split Levi subgroups $\bL$ of
   $\bG$, $\tQLG$ is given in Tables~\ref{tab:nsplit_2a3}--\ref{nsplit_a1q2}
   for large enough $q$.
 \end{enumerate}
 Representatives for all the unipotent classes are given in the Steinberg
 presentation in Tables~\ref{tab:classes_G}--\ref{tab:classes_last_Levi}, the
 numbering is chosen such that the centre node of the Dynkin diagram is
 labelled 3.
\end{thm}

The tables for $\Spin_8^+(q)$ and for the Levi subgroups of type $A_1^3$ also
show the action of the triality automorphism on the unipotent classes.

\begin{proof}
Part (a) follows by explicit computation of the fusion of unipotent classes of
$\bG^F$ and an application of~(\ref{eq2}). This gives the 2-parameter Green
functions for $M_1$, $M_2$ and $L_1$. Thanks to Lemma~\ref{lem:prod} and
Proposition~\ref{prop:A3} we also obtain $\tQ_{\bL_2}^\bG$.

Analogously to Proposition~\ref{prop:A3} we can compute $\tQ_{\bL_1}^{\bM_2}$
(explicitly) and $\tQ_{\bL_3}^{\bM_2}$ (by scalar products). Again,
Lemma~\ref{lem:prod} now gives $\tQ_{\bL_3}^\bG$.

The situation is now graphically summarised in Figure~\ref{fig:Levi_Lattice}.
The Levi subgroups in solid boxes are those with known Green functions (which
are valid for all odd $q$), while those in dashed boxes have to be found and,
unfortunately, will be shown to be valid only for large enough $q$. The reason
is the following. Both $\tQ_{\bL_4}^\bG$ and $\tQ_{\bL_5}^\bG$ have 5 unknown
entries and in total we get 7 scalar products of the form
$$\big\langle R_{\bL_i}^\bG f_i,\, R_{\bL_j}^\bG f_j \big\rangle = 0
  \quad \text{for } i\neq j,$$
plus two norm equations. Since (as can be seen from
Figure~\ref{fig:Levi_Lattice}) equations of the form
$$\big\langle R_{\textbf{M}_i}^\bG f,\, R_{\bL_j}^\bG f_j \big\rangle = 0
  \qquad\text{for } L_j\nsubseteq M_i$$
do not provide new information, the only course of action that we see, at
present, is to use the above mentioned result of Digne, Lehrer and Michel. 

Setting $\tQLG(u,v)$ to 1 or 0 for pairs of regular unipotent elements
allows us to solve the system of equations. Then, the Lusztig restriction of
explicitly computed Gel'fand--Graev characters allows us to select the only
correct solution. We mention, however, that the norm equation now yields two
rational solutions. For $q>3$, only one of them satisfies
$|C_{\bL^F}(v)|\tQLG(u,v)\in \ZZ$ which has to hold by the definition of the
modified Green functions.

Thanks to the knowledge of $\tQ_{\bL_i}^\bG$ for $i=2,3,4$ we can repeatedly
use Lemma~\ref{lem:prod} to determine the remaining Green functions.

The unipotent conjugacy classes are computed explicitly by fusing the classes
of a maximal unipotent subgroup under elements of a maximally split torus and
elements of the Weyl group. The precise action of the triality automorphism
follows from the knowledge of the fusion of the unipotent classes.
\end{proof}

\begin{rem}
Since the publication of the first version of this document L\"ubeck
\cite{Lue20} used a theoretically more involved method to also compute
2-parameter Green functions.
In short, he uses the Springer correspondence to compute some generalised Green
functions. Then he uses the fact that Lusztig induction of a generalised Green
function is a generalised Green function (with some conditions). This provides
him with the equations that are missing when considering only the ordinary Green
functions, like here.

With this he finds the same result as in our Tables~\ref{tab:split_a1a1a1}
and~\ref{tab:nsplit_a1a1a1_qp1}. However, as he states, with his method the
"large $q$" requirement can be dropped, so thanks to Lemma~\ref{lem:prod} all
of our tables are valid for arbitrary (odd)
$q$.
\end{rem}

\noindent
{\bf Acknowledgement}: We thank Frank L\"ubeck for pointing out that it might be
possible to explicitly compute 2-parameter Green functions using results on
ordinary Green functions and Jean Michel for alerting us to \cite{DiMi0,DiMi1}.


\begin{table}[bh]
\caption{$\tQLG(u,v)$ for the split Levi subgroup $A_3(q)(q-1)$ of $\Spin_8^+(q)$. \label{tab:split_a3}}
$$\begin{array}{c|cccccccccccc}
 & u_1 & u_2 & u_3 & u_4 & u_5 & u_6 & u_7 & u_8 & u_9 & u_{10} & u_{11} & u_{12} \\
 \hline
v_1 & \Phi_2^2\Phi_4\Phi_6 & \Phi_2 & . & . & . & . & 1 & 1 & . & . & . & . \\
v_2 & . & q^2\Phi_2^2 & \Phi_2\Phi_4 & \Phi_2\Phi_4 & \Phi_2\Phi_4 & \Phi_2\Phi_4 & . & . & 1 & 1 & 1 & 1 \\
v_3 & . & . & . & . & . & . & q\Phi_2\Phi_4 & . & q\Phi_2 & . & q\Phi_2 & . \\
v_4 & . & . & . & . & . & . & . & q\Phi_2\Phi_4 & . & q\Phi_2 & . & q\Phi_2 \\
v_5 & . & . & . & . & . & . & . & . & . & . & . & . \\
v_6 & . & . & . & . & . & . & . & . & . & . & . & . \\
v_7 & . & . & . & . & . & . & . & . & . & . & . & . 
\end{array}$$

$$\begin{array}{c|cccccccccccccccc}
 & u_{13} & u_{14} & u_{15} & u_{16} & u_{17} & u_{18} & u_{19} & u_{20} & u_{21} & u_{22} & u_{23} & u_{24} & u_{25} & u_{26} & u_{27} & u_{28} \\
 \hline
v_1 & . & . & . & . & . & . & . & . & . & . & . & . & . & . & . & . \\
v_2 & 2 & . & . & . & . & . & . & . & . & . & . & . & . & . & . & . \\
v_3 & \frac{\Phi_1}{2} & \frac{\Phi_2}{2} & . & . & . & . & 1 & . & . & . & . & . & . & . & . & . \\
v_4 & \frac{\Phi_1}{2} & \frac{\Phi_2}{2} & . & . & . & . & . & 1 & . & . & . & . & . & . & . & . \\
v_5 & 2q^2 & . & \Phi_2 & \Phi_2 & \Phi_2 & \Phi_2 & . & . & 1 & 1 & 1 & 1 & . & . & . & . \\ \cline{10-13}
v_6 & . & . & . & . & . & . & q\Phi_2 & . & \multicolumn{4}{|c|}{\multirow{2}{*}{\Large{$\ast$}}} & 1 & . & 1 & .\\
v_7 & . & . & . & . & . & . & . & q\Phi_2 & \multicolumn{4}{|c|}{} & . & 1 & . & 1 \\ \cline{10-13}
\end{array}$$
\vskip 1pc
\begin{minipage}{0.5\textwidth}
\centering{ $q\equiv1\pmod4$ }
$$\begin{array}{c|cccc}
 & u_{21} & u_{22} & u_{23} & u_{24} \\ \hline
v_6 & q & . & q & . \\
v_7 & . & q & . & q 
\end{array}$$
\end{minipage}
\begin{minipage}{0.49\textwidth}

\centering{ $q\equiv3\pmod4$ }
$$\begin{array}{c|cccc}
 & u_{21} & u_{22} & u_{23} & u_{24} \\ \hline
v_6 & . & q & . & q \\
v_7 & q & . & q & . 
\end{array}$$
\end{minipage}
\end{table}

\begin{table}[ht]
\tiny
\caption{$\tQLG(u,v)$ for the split Levi subgroup $A_1(q)^3(q-1)$ of $\Spin_8^+(q)$. \label{tab:split_a1a1a1}}
$$\begin{array}{c|cccccccccccc}
 & u_1 & u_2 & u_3 & u_4 & u_5 & u_6 & u_7 & u_8 & u_9 & u_{10} & u_{11} & u_{12} \\
 \hline
v_1 & \Phi_2\Phi_3\Phi_4^2\Phi_6 & * & \Phi_2\Phi_4 & \Phi_2\Phi_4 & \Phi_2\Phi_4 & \Phi_2\Phi_4 & \Phi_2\Phi_4 & \Phi_2\Phi_4 & \Phi_2 & \Phi_2 & \Phi_2 & \Phi_2 \\
v_2 & . & q^4\Phi_2 & . & . & . & . & q^2\Phi_4 & q^2\Phi_4 & . & . & . & . \\
v_3 & . & q^4\Phi_2 & . & . & q^2\Phi_4 & q^2\Phi_4 & . & . & . & . & . & . \\
v_4 & . & q^4\Phi_2 & q^2\Phi_4 & q^2\Phi_4 & . & . & . & . & . & . & . & . \\
v_5 & . & . & q^2\Phi_2\Phi_4 & . & . & . & . & . & q^2 & . & . & q^2 \\
v_6 & . & . & . & q^2\Phi_2\Phi_4 & . & . & . & . & . & q^2 & q^2 & . \\
v_7 & . & . & . & . & q^2\Phi_2\Phi_4 & . & . & . & q^2 & q^2 & . & . \\
v_8 & . & . & . & . & . & q^2\Phi_2\Phi_4 & . & . & . & . & q^2 & q^2 \\
v_9 & . & . & . & . & . & . & q^2\Phi_2\Phi_4 & . & q^2 & . & q^2 & . \\
v_{10} & . & . & . & . & . & . & . & q^2\Phi_2\Phi_4 & . & q^2 & . & q^2 \\
v_{11} & . & . & . & . & . & . & . & . & q^3\Phi_2 & . & . & . \\
v_{12} & . & . & . & . & . & . & . & . & . & q^3\Phi_2 & . & . \\
v_{13} & . & . & . & . & . & . & . & . & . & . & q^3\Phi_2 & . \\
v_{14} & . & . & . & . & . & . & . & . & . & . & . & q^3\Phi_2
\end{array}$$

$$\begin{array}{c|cccccccccccccccc}
 & u_{13} & u_{14} & u_{15} & u_{16} & u_{17} & u_{18} & u_{19} & u_{20} & u_{21} & u_{22} & u_{23} & u_{24} & u_{25} & u_{26} & u_{27} & u_{28} \\
 \hline
v_1 & 1 & 1 & . & . & . & . & . & . & . & . & . & . & . & . & . & . \\
v_2 & 2q & . & . & . & . & . & 1 & 1 & . & . & . & . & . & . & . & . \\
v_3 & 2q & . & . & . & 1 & 1 & . & . & . & . & . & . & . & . & . & . \\
v_4 & 2q & . & 1 & 1 & . & . & . & . & . & . & . & . & . & . & . & . \\
v_5 & q\Phi_1 & . & \Phi_2 & . & . & . & . & . & 1 & 1 & . & . & . & . & . & . \\
v_6 & q\Phi_1 & . & . & \Phi_2 & . & . & . & . & . & . & 1 & 1 & . & . & . & .\\
v_7 & q\Phi_1 & . & . & . & \Phi_2 & . & . & . & 1 & . & . & 1 & . & . & . & . \\
v_8 & q\Phi_1 & . & . & . & . & \Phi_2 & . & . & . & 1 & 1 & . & . & . & . & . \\ \cline{10-13}
v_9 & q\Phi_1 & . & . & . & . & . & \Phi_2 & . & \multicolumn{4}{|c|}{\multirow{7}{*}{\Large{$\ast$}}} & . & . & . & . \\
v_{10} & q\Phi_1 & . & . & . & . & . & . & \Phi_2 & \multicolumn{4}{|c|}{} & . & . & . & . \\
v_{11} & \frac{q\Phi_1^2}{4} & \frac{q\Phi_2^2}{4} & \frac{\Phi_1\Phi_2}{2} & . & \frac{\Phi_1\Phi_2}{2} & . & \frac{\Phi_1\Phi_2}{2} & . & \multicolumn{4}{|c|}{} & 1 & . & . & . \\
v_{12} & \frac{q\Phi_1^2}{4} & \frac{q\Phi_2^2}{4} & . & \frac{\Phi_1\Phi_2}{2} & \frac{\Phi_1\Phi_2}{2} & . & . & \frac{\Phi_1\Phi_1}{2} & \multicolumn{4}{|c|}{} & . & 1 & . & . \\
v_{13} & \frac{q\Phi_1^2}{4} & \frac{q\Phi_2^2}{4} & . & \frac{\Phi_1\Phi_2}{2} & . & \frac{\Phi_1\Phi_2}{2} & \frac{\Phi_1\Phi_2}{2} & . & \multicolumn{4}{|c|}{} & . & . & 1 & . \\
v_{14} & \frac{q\Phi_1^2}{4} & \frac{q\Phi_2^2}{4} & \frac{\Phi_1\Phi_2}{2} & . & . & \frac{\Phi_1\Phi_2}{2} & . & \frac{\Phi_1\Phi_2}{2} & \multicolumn{4}{|c|}{} & . & . & . & 1 \\ \cline{10-13}
\end{array}$$

$*=q^4+3q^3+3q^2+q+1$
\vskip 1pc
\begin{minipage}{0.5\textwidth}
\centering{ $q\equiv1\pmod4$ }
$$\begin{array}{c|cccc}
 & u_{21} & u_{22} & u_{23} & u_{24} \\ \hline
v_9    & 1 & . & 1 & . \\
v_{10} & . & 1 & . & 1 \\
v_{11} & \frac{q-5}{4} & \frac{\Phi_1}{4} & \frac{\Phi_1}{4} & \frac{\Phi_1}{4} \\
v_{12} & \frac{\Phi_1}{4} & \frac{\Phi_1}{4} & \frac{\Phi_1}{4} & \frac{q-5}{4} \\
v_{13} & \frac{\Phi_1}{4} & \frac{\Phi_1}{4} & \frac{q-5}{4} & \frac{\Phi_1}{4} \\
v_{14} & \frac{\Phi_1}{4} & \frac{q-5}{4} & \frac{\Phi_1}{4} & \frac{\Phi_1}{4} 
\end{array}$$
\end{minipage}
\begin{minipage}{0.49\textwidth}

\centering{ $q\equiv3\pmod4$ }
$$\begin{array}{c|cccc}
 & u_{21} & u_{22} & u_{23} & u_{24} \\ \hline
v_9    & . & 1 & . & 1 \\
v_{10} & 1 & . & 1 & . \\
v_{11} & \frac{q-3}{4} & \frac{q-3}{4} & \frac{\Phi_2}{4} & \frac{q-3}{4} \\
v_{12} & \frac{q-3}{4} & \frac{\Phi_2}{4} & \frac{q-3}{4} & \frac{q-3}{4} \\
v_{13} & \frac{\Phi_2}{4} & \frac{q-3}{4} & \frac{q-3}{4} & \frac{q-3}{4} \\
v_{14} & \frac{q-3}{4} & \frac{q-3}{4} & \frac{q-3}{4} & \frac{\Phi_2}{4} 
\end{array}$$
\end{minipage}

\end{table}

\begin{table}[ht]
\caption{(Modified) 2-parameter Green functions for $\tw2A_3(q)(q+1)$ for $q$ large. \label{tab:nsplit_2a3}}
\scriptsize
$$\begin{array}{c|cccccccccccc}
 & u_{1} & u_{2} & u_{3} & u_{4} & u_{5} & u_{6} & u_{7} & u_{8} & u_{9} & u_{10} & u_{11} & u_{12}  \\\hline
v_1 & \Phi_1^2\Phi_3\Phi_4 & -\Phi_1 & . & . & . & . & 1 & 1 & . & . & . & . \\
v_2 & . & q^2\Phi_1^2 & -\Phi_1\Phi_4 & -\Phi_1\Phi_4 & -\Phi_1\Phi_4 & -\Phi_1\Phi_4 & . & . & 1 & 1 & 1 & 1 \\
v_3 & . & . & . & . & . & . & q\Phi_1\Phi_4 & . & q\Phi_1 & . & q\Phi_1 & . \\
v_4 & . & . & . & . & . & . & . & q\Phi_1\Phi_4 & . & q\Phi_1 & . & q\Phi_1 \\
v_5 & . & . & . & . & . & . & . & . & . & . & . & . \\
v_6 & . & . & . & . & . & . & . & . & . & . & . & . \\
v_7 & . & . & . & . & . & . & . & . & . & . & . & . \\
\end{array}$$

$$\begin{array}{c|cccccccccccccccc}
& u_{13} & u_{14} & u_{15} & u_{16} & u_{17} & u_{18} & u_{19} & u_{20} & u_{21} & u_{22} & u_{23} & u_{24} & u_{25} & u_{26} & u_{27} & u_{28} \\\hline
v_1 & . & . & . & . & . & . & . & . & . & . & . & . & . & . & . & . \\
v_2 & . & 2 & . & . & . & . & . & . & . & . & . & . & . & . & . & . \\
v_3 & -\frac{\Phi_1}{2} & -\frac{\Phi_2}{2} & . & . & . & . & . & 1 & . & . & . & . & . & . & . & . \\
v_4 & -\frac{\Phi_1}{2} & -\frac{\Phi_2}{2} & . & . & . & . & 1 & . & . & . & . & . & . & . & . & . \\
v_5 & . & 2q^2 & -\Phi_1 & -\Phi_1 & -\Phi_1 & -\Phi_1 & . & . & 1 & 1 & 1 & 1 & . & . & . & . \\ \cline{10-13}
v_6 & . & . & . & . & . & . & q\Phi_1 & . & \multicolumn{4}{|c|}{\multirow{2}{*}{\Large{$\ast$}}} & . & 1 & . & 1 \\
v_7 & . & . & . & . & . & . & . & q\Phi_1 & \multicolumn{4}{|c|}{} & 1 & . & 1 & . \\ \cline{10-13}
\end{array}$$
\vskip 1pc
\begin{minipage}{0.5\textwidth}
\centering{ $q\equiv1\pmod4$ }
$$\begin{array}{c|cccc}
 & u_{21} & u_{22} & u_{23} & u_{24} \\ \hline
v_6 & -q & . & -q & . \\
v_7 & . & -q & . & -q 
\end{array}$$
\end{minipage}
\begin{minipage}{0.49\textwidth}

\centering{ $q\equiv3\pmod4$ }
$$\begin{array}{c|cccc}
 & u_{21} & u_{22} & u_{23} & u_{24} \\ \hline
v_6 & . & -q & . & -q \\
v_7 & -q & . & -q & . 
\end{array}$$
\end{minipage}
\end{table}

\begin{table}[ht]
\caption{(Modified) 2-parameter Green functions for $A_1(q)^3(q+1)$ for $q$ large. \label{tab:nsplit_a1a1a1_qp1}}
\scriptsize
$$\begin{array}{c|cccccccc}
 & u_{1} & u_{2} & u_{3} & u_{4} & u_{5} & u_{6} & u_{7} & u_{8} \\\hline
v_1 & -\Phi_1\Phi_3\Phi_4^2\Phi_6 & * & -\Phi_1\Phi_4 & -\Phi_1\Phi_4 & -\Phi_1\Phi_4 & -\Phi_1\Phi_4 & -\Phi_1\Phi_4 & -\Phi_1\Phi_4 \\
v_2 & . & -q^4\Phi_1 & . & . & . & . & q^2\Phi_4 & q^2\Phi_4 \\
v_3 & . & -q^4\Phi_1 & . & . & q^2\Phi_4 & q^2\Phi_4 & . & . \\
v_4 & . & -q^4\Phi_1 & q^2\Phi_4 & q^2\Phi_4 & . & . & . & . \\
v_5 & . & . & -q^2\Phi_1\Phi_4 & . & . & . & . & . \\
v_6 & . & . & . & -q^2\Phi_1\Phi_4 & . & . & . & . \\
v_7 & . & . & . & . & -q^2\Phi_1\Phi_4 & . & . & . \\
v_8 & . & . & . & . & . & -q^2\Phi_1\Phi_4 & . & . \\
v_9 & . & . & . & . & . & . & -q^2\Phi_1\Phi_4 & . \\
v_{10} & . & . & . & . & . & . & . & -q^2\Phi_1\Phi_4 \\
v_{11} & . & . & . & . & . & . & . & . \\
v_{12} & . & . & . & . & . & . & . & . \\
v_{13} & . & . & . & . & . & . & . & . \\
v_{14} & . & . & . & . & . & . & . & . \\
\end{array}$$

$$\begin{array}{c|cccccccccccc}
 & u_{9} & u_{10} & u_{11} & u_{12} & u_{13} & u_{14} & u_{15} & u_{16} & u_{17} & u_{18} & u_{19} & u_{20} \\\hline
v_1 & -\Phi_1 & -\Phi_1 & -\Phi_1 & -\Phi_1 & 1 & 1 & . & . & . & . & . & . \\
v_2 & . & . & . & . & . & -2q & . & . & . & . & 1 & 1 \\
v_3 & . & . & . & . & . & -2q & . & . & 1 & 1 & . & . \\
v_4 & . & . & . & . & . & -2q & 1 & 1 & . & . & . & . \\
v_5 & q^2 & . & . & q^2 & . & q\Phi_2 & . & -\Phi_1 & . & . & . & . \\
v_6 & . & q^2 & q^2 & . & . & q\Phi_2 & -\Phi_1 & . & . & . & . & . \\
v_7 & q^2 & q^2 & . & . & . & q\Phi_2 & . & . & . & -\Phi_1 & . & . \\
v_8 & . & . & q^2 & q^2 & . & q\Phi_2 & . & . & -\Phi_1 & . & . & . \\
v_9 & q^2 & . & q^2 & . & . & q\Phi_2 & . & . & . & . & . & -\Phi_1 \\
v_{10} & . & q^2 & . & q^2 & . & q\Phi_2 & . & . & . & . & -\Phi_1 & . \\
v_{11} & q^3\Phi_1 & . & . & . & -\frac{q\Phi_1^2}{4} & -\frac{q\Phi_2^2}{4} & . & \frac{\Phi_1\Phi_2}{2} & . & \frac{\Phi_1\Phi_2}{2} & . & \frac{\Phi_1\Phi_2}{2} \\
v_{12} & -\frac{q\Phi_1^2}{4} & -\frac{q\Phi_2^2}{4} & \frac{\Phi_1\Phi_2}{2} & . & . & \frac{\Phi_1\Phi_2}{2} & \frac{\Phi_1\Phi_2}{2} & . & . & q^3\Phi_1 & . & . \\
v_{13} & -\frac{q\Phi_1^2}{4} & -\frac{q\Phi_2^2}{4} & \frac{\Phi_1\Phi_2}{2} & . & \frac{\Phi_1\Phi_2}{2} & . & . & \frac{\Phi_1\Phi_2}{2} & . & . & q^3\Phi_1 & . \\
v_{14} & . & . & . & q^3\Phi_1 & -\frac{q\Phi_1^2}{4} & -\frac{q\Phi_2^2}{4} & . & \frac{\Phi_1\Phi_2}{2} & \frac{\Phi_1\Phi_2}{2} & . & \frac{\Phi_1\Phi_2}{2} & . \\
\end{array}$$
\vskip 1pc
\begin{minipage}{0.6\textwidth}
$$\begin{array}{c|cccccccc}
 & u_{21} & u_{22} & u_{23} & u_{24} & u_{25} & u_{26} & u_{27} & u_{28} \\\hline
v_1 & . & . & . & . & . & . & . & . \\
v_2 & . & . & . & . & . & . & . & . \\
v_3 & . & . & . & . & . & . & . & . \\
v_4 & . & . & . & . & . & . & . & . \\
v_5 & . & . & 1 & 1 & . & . & . & . \\
v_6 & 1 & 1 & . & . & . & . & . & . \\
v_7 & . & 1 & 1 & . & . & . & . & . \\
v_8 & 1 & . & . & 1 & . & . & . & . \\ \cline{2-5}
v_9 & \multicolumn{4}{|c|}{\multirow{6}{*}{\Large{$\ast$}}} & . & . & . & . \\
v_{10} & \multicolumn{4}{|c|}{} & . & . & . & . \\
v_{11} & \multicolumn{4}{|c|}{} & 1 & . & . & . \\
v_{12} & \multicolumn{4}{|c|}{} & . & 1 & . & . \\
v_{13} & \multicolumn{4}{|c|}{} & . & . & 1 & . \\
v_{14} & \multicolumn{4}{|c|}{} & . & . & . & 1 \\ \cline{2-5}
\end{array}$$
$*=q^4-3q^3+3q^2-q+1$
\end{minipage}
\begin{minipage}{0.39\textwidth}
\centering{ $q\equiv1\pmod4$ }
$$\begin{array}{c|cccc}
 & u_{21} & u_{22} & u_{23} & u_{24} \\ \hline
v_9 & . & 1 & . & 1 \\
v_{10} & 1 & . & 1 & . \\
v_{11} & -\frac{\Phi_1}{4} & -\frac{q+3}{4} & -\frac{q+3}{4} & -\frac{q+3}{4} \\
v_{12} & -\frac{q+3}{4} & -\frac{q+3}{4} & -\frac{q+3}{4} & -\frac{\Phi_1}{4} \\
v_{13} & -\frac{q+3}{4} & -\frac{q+3}{4} & -\frac{\Phi_1}{4} & -\frac{q+3}{4} \\
v_{14} & -\frac{q+3}{4} & -\frac{\Phi_1}{4} & -\frac{q+3}{4} & -\frac{q+3}{4} \\
\end{array}$$

\centering{ $q\equiv3\pmod4$ }
$$\begin{array}{c|cccc}
 & u_{21} & u_{22} & u_{23} & u_{24} \\ \hline
v_9 & 1 & . & 1 & . \\
v_{10} & . & 1 & . & 1 \\
v_{11} & -\frac{\Phi_2}{4} & -\frac{\Phi_2}{4} & -\frac{q+5}{4} & -\frac{\Phi_2}{4} \\
v_{12} & -\frac{\Phi_2}{4} & -\frac{q+5}{4} & -\frac{\Phi_2}{4} & -\frac{\Phi_2}{4} \\
v_{13} & -\frac{q+5}{4} & -\frac{\Phi_2}{4} & -\frac{\Phi_2}{4} & -\frac{\Phi_2}{4} \\
v_{14} & -\frac{\Phi_2}{4} & -\frac{\Phi_2}{4} & -\frac{\Phi_2}{4} & -\frac{q+5}{4} \\
\end{array}$$
\end{minipage}
\end{table}


\begin{table}[ht]
\caption{(Modified) 2-parameter Green functions for $A_1(q^2)(q^2+1)$ for $q$ large. \label{nsplit_a1q2}}
\scriptsize
$$\begin{array}{c|cccccccc}
 & u_1 & u_2 & u_3 & u_4 & u_5 & u_6 & u_7 & u_8 \\
 \hline
v_1 & -\Phi_1^3\Phi_2^3\Phi_3\Phi_6 &  \Phi_1^2\Phi_2^2 &  \Phi_1^2\Phi_2^2 &  \Phi_1^2\Phi_2^2 &  \Phi_1^2\Phi_2^2 &  \Phi_1^2\Phi_2^2 &  -\Phi_1\Phi_2 &  -\Phi_1\Phi_2  \\
v_2 & . & . & . & . & . & . & -q^2\Phi_1^2\Phi_2^2 & . \\
v_3 & . & . & . & . & . & . & . & -q^2\Phi_1^2\Phi_2^2 \\
\end{array}$$

$$\begin{array}{c|cccccccccccc}
 & u_9 & u_{10} & u_{11} & u_{12} & u_{13} & u_{14} & u_{15} & u_{16} & u_{17} & u_{18} & u_{19} & u_{20}\\
 \hline
v_1 &  -\Phi_1\Phi_2 &  -\Phi_1\Phi_2 &  -\Phi_1\Phi_2 &  -\Phi_1\Phi_2 &  -\Phi_1 &  \Phi_2 &  . &  . &  . &  . &  1 &  1 \\
v_2 & q^2\Phi_1\Phi_2 & . & q^2\Phi_1\Phi_2 & . &  \frac{q\Phi_1^2}{2} &  -\frac{q\Phi_2^2}{2} & -\frac{\Phi_1\Phi_2}{2} & -\frac{\Phi_1\Phi_2}{2} & -\frac{\Phi_1\Phi_2}{2} & -\frac{\Phi_1\Phi_2}{2} & . & . \\
v_3 & . & q^2\Phi_1\Phi_2 & . & q^2\Phi_1\Phi_2 &  \frac{q\Phi_1^2}{2} &  -\frac{q\Phi_2^2}{2} & -\frac{\Phi_1\Phi_2}{2} & -\frac{\Phi_1\Phi_2}{2} & -\frac{\Phi_1\Phi_2}{2} & -\frac{\Phi_1\Phi_2}{2} & . & . \\
\end{array}$$

$$\begin{array}{c|cccccccc}
 & u_{21} & u_{22} & u_{23} & u_{24} & u_{25} & u_{26} & u_{27} & u_{28} \\
 \hline
v_1 &  . &  . &  . &  . &  . &  . &  . &  . \\ \cline{2-5}
v_2 & \multicolumn{4}{|c|}{\multirow{2}{*}{\Large{$\ast$}}} & . &  1 & . &  1 \\
v_3 & \multicolumn{4}{|c|}{} &  1 & . &  1 & . \\ \cline{2-5}
\end{array}$$
\vskip 1pc
\begin{minipage}{0.5\textwidth}
\centering{ $q\equiv1\pmod4$ }
$$\begin{array}{c|cccc}
 & u_{21} & u_{22} & u_{23} & u_{24} \\ \hline
v_2 & \frac{\Phi_2}{2} & -\frac{\Phi_1}{2} & \frac{\Phi_2}{2} & -\frac{\Phi_1}{2} \\
v_3 & -\frac{\Phi_1}{2} & \frac{\Phi_2}{2} & -\frac{\Phi_1}{2} & \frac{\Phi_2}{2} 
\end{array}$$
\end{minipage}
\begin{minipage}{0.49\textwidth}

\centering{ $q\equiv3\pmod4$ }
$$\begin{array}{c|cccc}
 & u_{21} & u_{22} & u_{23} & u_{24} \\ \hline
v_2 & -\frac{\Phi_1}{2} & \frac{\Phi_2}{2} & -\frac{\Phi_1}{2} & \frac{\Phi_2}{2} \\
v_3 & \frac{\Phi_2}{2} & -\frac{\Phi_1}{2} & \frac{\Phi_2}{2} & -\frac{\Phi_1}{2} 
\end{array}$$
\end{minipage}
\end{table}


\begin{table}[ht]
\caption{Unipotent classes of $\Spin_8^+(q)$. The table shows the Jordan normal form, the numbering of the class, a representative in Steinberg presentation ($\left\langle\mu\right\rangle=\FF_q^\times$) and the image under the triality automorphism. On the classes with Jordan form $53$ the triality automorphism acts differently depending on the congruence of $q$. The action for $q\equiv1\,(4)$ is given first. \label{tab:classes_G}}
$\begin{array}{c|c|l|c}
1^8 & u_1 & 1 & u_1 \\ \hline

2^21^4 & u_2 & u_1(1) & u_2 \\ \hline

2^4+ & u_3 & u_1(1)u_2(1) & u_7 \\ 
 & u_4 & u_1(\mu)u_2(1) & u_8 \\ \hline

2^4- & u_5 & u_1(1)u_4(1) & u_3 \\ 
 & u_6 & u_1(\mu)u_4(1) & u_4 \\ \hline

31^5 & u_7 & u_2(1)u_4(1) & u_5 \\ 
 & u_8 & u_2(\mu)u_4(1) & u_6 \\ \hline

32^21 & u_9 & u_1(1)u_2(1)u_4(1) & u_9 \\ 
 & u_{10} & u_1(1)u_2(\mu)u_4(1) & u_{12} \\
 & u_{11} & u_1(\mu)u_2(1)u_4(1) & u_{10} \\ 
 & u_{12} & u_1(\mu)u_2(\mu)u_4(1) & u_{11} \\ \hline

3^21^2 & u_{13} & u_1(1)u_3(1) & u_{13} \\ \hline

3^21^2 & u_{14} & u_1(\mu)u_2(1)u_4(1)u_{12}(1) & u_{14} \\ \hline

4^2+ & u_{15} & u_1(1)u_2(1)u_3(1) & u_{19} \\
 & u_{16} & u_1(\mu)u_2(1)u_3(1) & u_{20} \\ \hline

4^2- & u_{17} & u_1(1)u_3(1)u_4(1) & u_{15} \\
 & u_{18} & u_1(\mu)u_3(1)u_4(1) & u_{16} \\ \hline

51^3 & u_{19} & u_2(1)u_3(1)u_4(1) & u_{17} \\
 & u_{20} & u_2(\mu)u_3(1)u_4(1) & u_{18} \\ \hline

53 & u_{21} & u_1(1)u_2(1)u_3(1)u_{10}(1) & u_{21}/u_{22} \\
 & u_{22} & u_1(1)u_2(1)u_3(\mu)u_{10}(1) & u_{23}/u_{24} \\
 & u_{23} & u_1(\mu)u_2(1)u_3(1)u_{10}(1) & u_{24}/u_{23} \\
 & u_{24} & u_1(\mu)u_2(1)u_3(\mu)u_{10}(1) & u_{22}/u_{21} \\ \hline

71 & u_{25} & u_1(1)u_2(1)u_3(1)u_4(1) & u_{25} \\
 & u_{26} & u_1(1)u_2(\mu)u_3(1)u_4(1) & u_{28} \\
 & u_{27} & u_1(\mu)u_2(1)u_3(1)u_4(1) & u_{26} \\
 & u_{28} & u_1(\mu)u_2(\mu)u_3(1)u_4(1) & u_{27} \\
\end{array}$
\end{table}

\begin{table}[ht]
\caption{Representatives of the unipotent classes of the Levi subgroups $A_3(q)(q-1)$ and $\tw2A_3(q)(q+1)$, where $\left\langle\rho\right\rangle=\FF_{q^2}^\times$ and $\mu=\rho^{q+1}$. \label{tab:classes_non_split_Levis}}
$\begin{array}{c|l}
v_1 & 1 \\ \hline

v_2 & u_2(1) \\ \hline

v_3 & u_2(1)u_4(1) \\
v_4 & u_2(\mu)u_4(1)\\ \hline

v_5 & u_2(1)u_3(1)\\ \hline

v_{6} & u_2(1)u_3(1)u_4(1) \\ 
v_7 & u_2(\mu)u_3(1)u_4(1) 
\end{array}$
\hspace{2cm}
$\begin{array}{c|l}
v_1 & 1 \\ \hline
v_2 & u_3(1) \\\hline
v_3 & u_2(1)u_4(1)\\ 
v_4 & u_2(\rho)u_4(\rho^q)\\ \hline
v_5 & u_2(\rho)u_4(\rho^q)u_6(1)u_7(1)\\ \hline
v_6 & u_2(1)u_4(1)u_3(1)\\ 
v_7 & u_2(\rho)u_4(\rho^q)u_3(1)\\ 
\end{array}$
\end{table}

\begin{table}[ht]
\caption{Representatives of the unipotent classes of the Levi subgroups $A_1(q)^3(q-1)$ and $A_1(q)^3(q+1)$ with the action of triality, where $\left\langle\mu\right\rangle=\FF_q^\times$. \label{tab:classes_split_Levis}}

$\begin{array}{c|l|c}
v_1 & 1 & v_1 \\ \hline

v_2 & u_1(1) & v_4 \\ \hline

v_3 & u_2(1) & v_2 \\ \hline

v_4 & u_4(1) & v_3 \\ \hline

v_5 & u_1(1)u_2(1) & v_9 \\ 
v_6 & u_1(\mu)u_2(1) & v_{10} \\ \hline

v_7 & u_1(1)u_4(1) & v_5 \\ 
v_8 & u_1(\mu)u_4(1) & v_6 \\ \hline

v_9 & u_2(1)u_4(1) & v_7 \\ 
v_{10} & u_2(\mu)u_4(1) & v_8 \\ \hline

v_{11} & u_1(1)u_2(1)u_4(1) & v_{11} \\ 
v_{12} & u_1(1)u_2(\mu)u_4(1) & v_{14} \\ 
v_{13} & u_1(\mu)u_2(1)u_4(1) & v_{12} \\ 
v_{14} & u_1(\mu)u_2(\mu)u_4(1) & v_{13} 
\end{array}$
\vskip 2pc
\end{table}

\begin{table}[ht]
\caption{Representatives of the unipotent classes of a non-split Levi subgroup of type $A_1(q^2)(q^2+1)$, where  $\left\langle\rho\right\rangle=\FF_{q^2}^\times$. \label{tab:classes_last_Levi}}
$\begin{array}{c|l}
v_1 & 1 \\ \hline
v_2 & u_2(1)u_4(1) \\
v_3 & u_2(\rho)u_4(\rho^q)\\ 
\end{array}$
\end{table}



\end{document}